\documentclass[11pt]{article}
\usepackage{epsfig,amsmath,latexsym,amssymb}
\usepackage{amsthm}
\usepackage{amscd }
\usepackage{dsfont}
\usepackage{multirow}
\usepackage{graphicx}
\usepackage{lscape}
\usepackage{subcaption}
\usepackage{color}
\usepackage[round]{natbib}
\usepackage{epstopdf}
\usepackage{bbm}
\usepackage{mathtools}
\usepackage{enumerate}
\def\R{{\mathbb R}}  
\def\E{{\mathbb E}}  
\def\F{{\mathcal F}}  
\def\P{{\mathbb P}}

\renewcommand{\epsilon}{\varepsilon}
\newcommand{\BH}{B^H}
\newcommand{\myset}{|\BH(t)+W(t)+g(t)| \leq \epsilon f(t),~ 0 \leq t \leq T}
\newcommand{\1}{\mathbbm{1}}
\newcommand{\D}{D}

\newcommand{\Remm}[1]{}
\newtheorem{theo}{Theorem}[section]
\newtheorem{lemma}[theo]{Lemma}

\newtheorem{cor}[theo]{Corollary}
\newtheorem{defi}[theo]{Definition}
\newtheorem{model ass}[theo]{Model}

\theoremstyle{remark}
\newtheorem{remark}{Remark}

\newcounter{example}
\newenvironment{example}[1][]{\refstepcounter{example}\par\medskip\noindent%
	\textbf{Example~\theexample. #1} \rmfamily}{\medskip}

\numberwithin{equation}{section}

\begin{document}

	
	
	
\title{Optimization of small deviation for  mixed fractional Brownian motion with trend\thanks{Published in \textit{Stochastics}, DOI: 10.1080/17442508.2018.1478835} \thanks{This work was supported by the Montreal Institute of Structured Finance and Derivatives [grant number R2091]. Anne MacKay also acknowledges support from the Natural Science and Engineering Council of Canada [grant number 04869].}}
	

\author{Anne MacKay\footnote{Department of Mathematics, Universit\'{e} du Qu\'{e}bec \`{a} Montr\'{e}al, Montreal, Canada.} \and Alexander Melnikov\footnote{Department of Mathematics, University of Alberta, Edmonton, Canada.} \and Yuliya Mishura\footnote{Department of Probability, Statistics and Actuarial Mathematics, Taras Shevschenko National University of Kyiv, Kyiv, Ukraine.} \thanks{Corresponding author. Email: myus@univ.kiev.ua} 
}

\maketitle

\begin{abstract}
	In this paper, we investigate two-sided bounds for the small ball probability of a mixed fractional Brownian motion with a general deterministic trend function, in terms of respective small ball probability of a mixed fractional Brownian motion without trend. To maximize the lower bound, we consider various ways to split the trend function between the components of the mixed fractional Brownian motion for the application of Girsanov theorem, and we show that the optimal split is the solution of a Fredholm integral equation. We find that the upper bound for the probability is also a function of this optimal split. The asymptotic behaviour of the probability as the ball becomes small is analyzed for zero trend function and for the particular choice of the upper limiting function.
\end{abstract}


\section{Introduction}
The present paper is devoted to two-sided bounds for the small ball probability of the form \begin{equation}\label{model-1}\mathbb{P}(|W(t)+B^H(t)+g(t)|\leq \varepsilon f(t),  0\leq t\leq T),\end{equation}
where $W$ is a Wiener process, $B^H$ is an independent  fractional Brownian motion with Hurst index $H<1/2$, $g$ is non-random trend, $f$ is some non-random measurable function (upper limit), and $\varepsilon\rightarrow 0.$ The solution of the general problem of small deviations for centered Gaussian processes $X$ with trend consists of three parts: (a) removal of the trend, (b) determination of the asymptotics in $\varepsilon$ for the centered Gaussian process, and (c) calculation of the precise constants in the asymptotic formulas. The problem has not yet been solved in all its generality, however, some special cases are considered.   For the determination of asymptotics in $\varepsilon$ when $f$ is a constant, see an interesting survey of \cite{li-shao}.  The particular case of the fractional Brownian motion was studied in \cite{shao-1993}, \cite{monrad-rootzen}, see also \cite{Li-Linde}. A survey of \cite{fatalov} is devoted to the calculation of the precise constants, also in the case when $f$ is a constant. An elegant method of calculation of the precise constant
in the case where $f$ is some smooth function and the Gaussian process $X$ is a Wiener process is presented in \cite{novikov}. The latter calculation is based on the martingale properties  of the Wiener process and integrals w.r.t. a Wiener process.
The approach of Novikov was extended to the case of a Gaussian martingale
$X$ and functions $g$ and $f$ which are absolutely continuous w.r.t. a
quadratic characteristic of $X$ in \cite{HM1} and \cite{HM2}. The
method developed in these papers consists of an appropriate change of
measure and some analytical properties of Gaussian processes with
independent increments.

The choice to study the small deviations of the Gaussian process $X=W+B^H$ is motivated by several reasons, and especially by the needs of financial models. On the one hand, at the moment, there is increasing evidence of memory and thus non-Markovian properties in observed return and  price volatility processes, and the behaviour of both prices and volatility is rough enough, which corresponds to a fractional Brownian component with Hurst index $H<1/2$. On the other hand, the possibility of stabilizing such a rough market in the long term is not ruled out, so the inclusion of the Wiener component is reasonable. We note that the mixed Brownian-fractional Brownian model $W+B^H$ was considered for the first time in \cite{cheridito} for the case $H>1/2$, in order to calculate the corresponding option prices in this  market. It was established that such market, unlike the market controlled by pure fBm, is arbitrage-free.    Additionally, the absence of arbitrage possibilities in such a market, in the class of Markov self-financing strategies, was established in \cite{mish-andr}. In the paper \cite{cheridito}, it was also proved that for $H>3/4$, $W+B^H$ is equivalent in measure to some Wiener process, and   generally speaking, the properties of such a mixed model are in many respects similar to properties of the Wiener component. A generalizing result was obtained in the article of \cite{vanzanten}, where it was proved that the linear combination of two independent fractional Brownian motions with Hurst indices $H_1$ and $H_2$ is equivalent in measure to the fBm with   a smaller Hurst index $H_1$ if the difference $H_2-H_1>1/4$. It was also shown that generally speaking, the properties of such a mixed model are in many respects similar to properties of the   component with a smaller Hurst index $H_1$. The corresponding Radon-Nikodym derivative was presented in terms of reproducing kernels. The Radon-Nikodym derivative for the case when  $H_1<1/4, H_2=1/2$ was presented also in \cite{cai}.

 Let  $(\Omega,\F,\{\F_t\}_{t\geq0},\mathbb{P})$  be a filtered probability space supporting all stochastic processes introduced below, and all of them are assumed to be adapted to this filtration. As it was mentioned, we consider a mixed Gaussian process composed of two independent processes: a Wiener process $W$ and a fractional Brownian motion (fBm) $B^H$ with Hurst index $H \in (0,1/2)$. Recall that a an fBm $B^H=\{B^H(t),t\geq 0\}$ is a centered Gaussian process with covariance function
\begin{align*}
	R_H(s,t) \coloneqq \E\left[B^H(s)B^H(t)\right]
	=\frac{1}{2}\left(t^{2H}+s^{2H}-|t-s|^{2H}\right),
	 \end{align*}
for $t, s \geq 0$. Our main goal is to get two-sided bounds for small ball probability presented in \eqref{model-1} as $\varepsilon\rightarrow 0$. The main step is to remove the trend $g$. At first we consider the lower bound. In contrast to the pure model, at this moment an interesting effect arises: we can share the trend  $g$ between the components $W$ and $B^H$, by letting $g=g_W+g_B$, apply Girsanov theorem for the mixed model, and this  leads to different lower bounds. Therefore we  can maximize the lower bound among different choices of trend sharing. It happens so that the optimal choice of the trend components is the solution of Fredholm integral equation of the second kind, and it is established that the equation has a unique solution. Note that the maximization of the lower bound depends only on the trend $g$ and does not involve the properties of function $f$, so a similar lower bound holds for the  probability
$$\mathbb{P}(|W(t)+B^H(t)+g(t)|\leq F(t),  0\leq t\leq T)$$ with any measurable non-random function $F$. Then we were lucky in the sense that it was possible to  apply smoothing procedure to the same expansion $g=g_W+g_B$ of the trend that maximizes the lower bound, and obtain an upper bound. The upper bound  depends on the function $\varepsilon f$ and ``works'' asymptotically as $\varepsilon \rightarrow 0.$ After all, we reduce the small ball probability for the mixed fractional Brownian motion with trend to the respective small ball probability for the mixed fractional Brownian motion without trend. Concerning this probability, its precise value is unknown, but  we give its asymptotics as $\varepsilon \rightarrow 0$.  Note that two-sided bounds for the probability $\P(B_H(t)+g(t)\leq f(t), t\geq 0)$  were obtained in \cite{hash-mish}.

The paper is organized as follows. Section \ref{sec:preliminaries} contains elements of fractional calculus, Wiener integration  w.r.t. fractional Brownian motion with $H < 1/2$ and  Girsanov theorem for fractional Brownian motion. We prove that for smooth functions, the integral w.r.t. a fBm with $H<1/2$, introduced as the integral w.r.t. to an underlying Wiener process with fractionally transformed integrand, coincides with the integral obtained via integration by parts. Section \ref{sec: small-ball} contains the main results: lower and upper bounds for small ball probability \eqref{model-1}. For the reader's convenience, all the results, both auxiliary and main statements, are formulated in this section, but the proofs of the deterministic  results obtained by methods of functional analysis,   are postponed to the Appendix. Section \ref{section-small} briefly describes  the asymptotics of  the small ball probabilities for the centered mixed processes.  The Appendix contains auxiliary  statements related to functional calculus and the proofs of such statements from  Section \ref{sec: small-ball}.

\section{Preliminaries}\label{sec:preliminaries} Let us introduce the necessary objects and their relationships.
\subsection{Basic elements of  fractional calculus} In this subsection  we review basic elements of Riemann-Liouville fractional calculus. For more details  on the topic, see \cite{samko1993fractional}.

\begin{defi}\label{def:RLint}
	The left-sided Riemann-Liouville fractional integral operator of order $\alpha$ over the interval $[0,T]$ is defined for $\alpha > 0$ and $T>0$ by
	\begin{equation*}
	(I_{0}^\alpha f)(t) = \frac{1}{\Gamma(\alpha)} \int_0^t (t-z)^{\alpha-1} f(z) dz,
	\qquad t \in [0,T],
	\end{equation*}
	where $\Gamma(\cdot)$ is the Euler gamma function. The right-sided integral operator of the same order on $[0,T]$ is defined by
	\begin{equation*}
	(I_{T }^\alpha f)(t) = \frac{1}{\Gamma(\alpha)} \int_t^T (z-t)^{\alpha-1} f(z) dz,
	\qquad t \in [0,T].
	\end{equation*}
\end{defi}

It is also  mentioned in \cite{samko1993fractional} that the fractional integrals  $I^\alpha_{0} f$ and $I^\alpha_{T} f$ exist if ${f \in L_1([0,T])}$.

For $p \geq 1$,   denote the classes
\begin{align*}
\mathcal{I}^{\alpha}_+(L_p([0,T])) &= \{f:f = I_{0}^\alpha \varphi \text{ for some } \varphi \in L_p([0,T]) \},\\
\mathcal{I}^{\alpha}_-(L_p([0,T])) &= \{f:f = I_{T^-}^\alpha \varphi \text{ for some } \varphi \in L_p([0,T]) \}.
\end{align*}

The corresponding left-side and right-side fractional derivatives are denoted for $0<\alpha< 1$ by
\begin{align*}
(I_{0}^{-\alpha} f)(t) :=(\D_{0}^\alpha f)(t) &:= \frac{1}{\Gamma(1-\alpha)}\frac{d}{dt}\left(\int_0^t (t-z)^{-\alpha} ~f(z) dz\right),\\
(I_{T}^{-\alpha} f)(t):=(\D_{T^-}^\alpha f)(t) &:= -\frac{1}{\Gamma(1-\alpha)} \frac{d}{dt}\left(\int_t^T (z-t)^{-\alpha} ~f(z) dz\right),
\end{align*}
respectively.

For $f \in \mathcal{I}^\alpha_{\pm}(L_p([0,T]))$, $p \geq 1, 0 \leq \alpha \leq 1$, we have
\begin{equation*}
I^{\alpha}_\pm \D^{\alpha}_\pm f = f.
\end{equation*}
 In order to consider the integral transformations of stochastic processes, we introduce the weighted fractional integral operators, which are defined as
\begin{align}\label{operators}
&(K^H_{0 }f)(t) = C_1 t^{H-1/2}\left(I^{H-1/2}_{0}u^{1/2-H}f(u)\right)(t),\nonumber\\
&(K^{H,*}_{0 }f)(t) = (C_1)^{-1} t^{H-1/2}\left(I^{1/2-H}_{0}u^{1/2-H}f(u)\right)(t),\nonumber\\
&(K^H_{T}f)(t) = C_1 t^{1/2-H}\left(I^{H-1/2}_{T^-}u^{H-1/2}f(u)\right)(t),\nonumber\\
&(K^{H,*}_{T}f)(t) = (C_1)^{-1} t^{1/2-H}\left(I^{1/2-H}_{T^-}u^{H-1/2}f(u)\right)(t),
\end{align}
with $C_1 = \left(\frac{2H\Gamma(H+1/2)\Gamma(3/2-H)}{\Gamma(2-2H)}\right)^{1/2}$. Note that $u$ is a ``dummy'' argument in \eqref{operators}, so for example, the first equality can be rewritten as $(K^H_{0}f)(t) = C_1 t^{H-1/2}\left(I^{H-1/2}_{0}\cdot^{1/2-H}f(\cdot)\right)(t)$.
Operators  $K^H_{T}$ and $K^{H,*}$ were initially introduced  in \cite{jost} and  the whole set \eqref{operators} was considered in \cite{hash-mish}. As it was established in \cite{jost}, for $f\in  L_2([0,T])$ and $H<1/2$
\begin{equation}\label{inverse} K^H_{T}K^{H,*}_{T}f=f.\end{equation}
In what follows we shall denote by $C$ various constants whose values are not important and that vary from line to line.
\begin{remark}\label{remark1} According to Hardy-Littlewood theorem, see, e.g.,  Theorem 3.5  from \cite{samko1993fractional}, roughly speaking, for $0 < \alpha < 1, 1 < p < 1/\alpha$, the fractional integrals $I_{\pm}^\alpha$  are bounded operators from
  $L_p$ into $L_q$ for  $q = p/(1- \alpha p)$. Consider $\alpha=1/2-H$ and $q=2$. In this case $p=\frac{1}{1-H}$. Obviously,  $\frac{1}{1-H}< 1/\alpha=\frac{1   }{1/2-H}$,  consequently, Hardy-Littlewood theorem is applicable. Moreover, let $\varphi\in L_{2}([0,T])$. Since for $H<1/2$ we have that $(1-H)^{-1}<2$, then $\varphi\in L_{\frac{1}{1-H}}([0,T])$. It means that for any function $\varphi\in L_{2}([0,T])$,  $(K^{H,*}_{0 }f)\in L_{2}([0,T])$. Indeed, $u^{1/2-H}\leq t^{1/2-H}$ for $u\leq t$, therefore
  \begin{equation*}\begin{gathered}\int_0^T((K^{H,*}_{0 }f)(t))^2dt=(C_1)^2\int_0^Tt^{2H-1 }\left(I^{1/2-H}_{0}\left(u^{1/2-H}f(u)\right)(t)\right)^2dt\\  \leq C\int_0^T(I^{1/2-H}_{0}|f|(t))^2dt\leq C\int_0^T  f^2(t)dt<\infty.\end{gathered}\end{equation*}
  The same is true, of course, for $K^{H,*}_{T }$ as well. Since $u^{H-1/2}\leq t^{H-1/2}$ for $u\geq t$,
  \begin{equation*}\begin{gathered}\int_0^T\left((K^{H,*}_{T }f)(t)\right)^2dt=(C_1)^2\int_0^Tt^{1-2H }\left(I^{1/2-H}_{0}\left(u^{H-1/2}f(u)\right)(t)\right)^2dt\\ \leq C\int_0^T\left(I^{1/2-H}_{0}|f|(t)\right)^2dt\leq C\int_0^T f^2(t)dt<\infty.\end{gathered}\end{equation*}
\end{remark}

\subsection{Fractional Brownian motion and related processes. Wiener integration w.r.t. fractional Brownian motion with $H<1/2$. Girsanov theorem for fBm}\label{subsec-2.3}
 The  results of this section, except Lemma \ref{boudvar} which, in our opinion, is   new,   come  from \cite{decreusefond1998fractional}, \cite{jost}, \cite{biagini2008stochastic}
 and  \cite{mishura2008stochastic}.
First, we recall the main connections between an fBm and the so-called underlying Wiener process.
\begin{theo}\label{theo-girs}Let $H<1/2$. Then the following statements hold:
\begin{itemize}
\item[$(i)$] Let $B^H=\{B^H(t),~ t\geq 0\}$ be a fractional Brownian motion. Then the process ${B =\{B(t),~ t\geq 0\}}$ defined by
   \begin{align*}B(t)
   &= \frac{\Gamma(3/2-H)}{\Gamma(H+1/2)}\int_0^t(K_T^{H,*}\1_{[0,t]})(s)dB^H(s)\\
   &=(1/2-H)(C_1)^{-1}\Gamma(3/2-H)\int_0^ts^{1/2-H}\int_s^tu^{H-1/2}
    (u-s)^{-1/2-H}du~dB^H(s)\end{align*}
    is a Wiener process that is called the underlying Wiener process.
  \item[$(ii)$] Let $B =\{B(t),~ t\geq 0\}$   be a Wiener process. Then the process $B^H=\{B^H(t),~ t\geq 0\}$ defined by
   \begin{equation*}\begin{gathered}B^H(t)= \int_0^t(K_T^{H}\1_{[0,t]})(s)dB(s)=\frac{C_1}{\Gamma(H+1/2)} \int_0^t\bigg(\left(\frac{t}{s}\right)^{H-1/2}(t-s)^{H-1/2}\\-
   \left(H-1/2\right) s^{1/2-H} \int_s^t (z-s)^{H-1/2}z^{H-3/2}dz\bigg)dB(s) \end{gathered}\end{equation*}
   is a fractional Brownian motion.
   \end{itemize}\end{theo}

Further, having a notion of the underlying Wiener process,
one  can  introduce and apply the Wiener integral w.r.t. an fBm as follows (see \cite{jost}).
For $T>0$ and $H \in (0,1/2)$, the Wiener integral w.r.t. an fBm is defined as
	\begin{align*}
	\int_0^T f(t) dB^H(t) = \int_0^T (K^H_T f)(t) dB(t).
	\end{align*}
	Taking into account \eqref{inverse}, the class $K_H([0,T])$ of admissible functions $f$ is described as
$$K_H([0,T])=\{f:[0,T]\rightarrow \R:\;\text{there exists}\;\varphi\in L^2([0,T])\;\text{such that }\; f=(K^{H,*}_{T}\varphi)\} \}.$$
 The class of elementary (step)  functions is dense in $K_H([0,T])$ w.r.t. the inner product defined by
  $(f,g)_{K_H([0,T])}=(K^H_T f, K^H_T g)_{L^2([0,T])}.$  Moreover, it is true that the space $K_H([0,T])$ is complete  and that the  operator
   $$J_T(f)=\int_0^T f(t) dB^H(t) = \int_0^T (K^H_T f)(t) dB(t)$$ is an isometry on  $K_H([0,T])$ if and only if $H<1/2.$ For these results see \cite{pipiras} and \cite{jost}.
   Now, let function $f$ be smooth, for example, let $f\in C^{(1)}([0,T])$, i.e., $f$ is continuously differentiable. Evidently, the Riemann integral $\int_0^TB^H(t)df(t)$ is well-defined. Let us prove that in this case  $\int_0^T f(t) dB^H(t)$ can be considered as the result of integration by parts.
   \begin{lemma}\label{boudvar} Let $f\in C^{(1)}([0,T])$. Then
   \begin{equation*} \int_0^T f(t) dB^H(t)=B^H(T)f(T)-\int_0^TB^H(t)df(t).
   \end{equation*}
   \end{lemma}
   \begin{proof} Consider the sequence  of dyadic partitions of the interval $[0,T]$:
   $$\pi_n=\{t_{k,n}=\frac{Tk}{2^n},\;0\leq k\leq 2^n\},$$ and denote $\Delta_n=\frac{T}{2^n}$. Choose the sequences of functions
   $$f_n(t)=\sum_{k=0}^{2^n-1} f(t_{k,n})\1_{(t_{k,n},t_{k+1,n}]}(t),\;q_n(t)=\sum_{k=0}^{2^n-1} f^\prime(t_{k,n})(t-t_{k,n})  \1_{(t_{k,n},t_{k+1,n}]}(t).$$
   Then $f_n, q_n$ are piece-wise differentiable, and
   $$\delta_n:=\sup_{t\in[0,T]\setminus\{\pi_n\}}|f^\prime(t)-q^\prime_n(t)|\leq \sup_{s,\; t\in[0,T],\;|s-t|\leq \Delta_n}|f^\prime(t)-f^\prime(s)|\rightarrow 0\;, n\rightarrow \infty.$$
   Furthermore,  $\epsilon_n:=\sup_{t\in[0,T]}|f(t)-f_n(t)|\rightarrow 0$ as $n\rightarrow \infty$, and additionally  the  variation admits the total bound $$V(f_n,[0,T])\leq V(f,[0,T])<\infty.$$ Let us write the formula for operator $K_T^H$:
   \begin{equation}\begin{gathered}\label{operator int}  K_T^H(f-f_n)(t) =\frac{C_1t^{1/2-H}}{\Gamma(H+1/2)} \frac{d}{dt}
   \left(\int_t^T(s-t)^{H-1/2}s^{H-1/2}(f(s)-f_n(s))ds\right).
   \end{gathered}\end{equation}
    Taking into account the fact that $f -f_n$ is piecewise differentiable and applying   integration by parts we get that
    \begin{equation}\begin{gathered}\label{operator int-1}\int_t^T(s-t)^{H-1/2}s^{H-1/2}(f(s)-f_n(s))ds
    =\frac{(T-t)^{H+1/2}}{H+1/2}T^{H-1/2}(f(T)-f_n(T))\\-\int_t^T\frac{(s-t)^{H+1/2}}{H+1/2}\left((H-1/2)s^{H-3/2}(f(s)-f_n(s))+
    s^{H-1/2}(f^\prime(s)-q^\prime_n(s))\right)ds\\-\int_t^T\frac{(s-t)^{H+1/2}}{H+1/2} s^{H-1/2}q^\prime_n(s)ds.
    \end{gathered}\end{equation}
    Substituting \eqref{operator int-1} into \eqref{operator int}, we get that \begin{equation}\begin{gathered}\label{operator-K}K_T^H(f-f_n)(t) = C t^{1/2-H}\bigg( (T-t)^{H-1/2} T^{H-1/2}(f(T)-f_n(T))\\+\int_t^T {(s-t)^{H-1/2}} \left((H-1/2)s^{H-3/2}(f(s)-f_n(s))+
    s^{H-1/2}(f^\prime(s)-q^\prime_n(s))\right)ds\\+\int_t^T (s-t)^{H-1/2}  s^{H-1/2}q^\prime_n(s)ds\bigg).
     \end{gathered}\end{equation}
 Note that $$\int_t^T (s-t)^{H-1/2}  s^{H-1/2}q^\prime_n(s)ds=\sum_{(t_{k,n},t_{k+1,n}]\cap(t,T]\neq \emptyset}\int_{(t_{k,n}\vee t,t_{k+1,n}]}(s-t)^{H-1/2}s^{H-1/2}(s-t_{k,n})ds.$$
  For example, if  $(t_{k,n},t_{k+1,n}]\subset(t,T]$, then
 \begin{equation*}\begin{gathered} \int_{(t_{k,n},t_{k+1,n}]}(s-t)^{H-1/2}s^{H-1/2}(s-t_{k,n})ds\\
 \leq C\int_{(t_{k,n},t_{k+1,n}]}(s-t_{k,n})^{H-1/2}(s-t_{k,n})^{H-1/2}(s-t_{k,n})ds\\ \leq C\int_{(t_{k,n},t_{k+1,n}]}(s-t_{k,n})^{2H}ds\leq C2^{-n(1+2H)}.\end{gathered}\end{equation*}
 Therefore
    \begin{equation}\begin{gathered}\label{up-up}|K_T^H(f-f_n)(t)| \leq C(\epsilon_n+\delta_n)t^{1/2-H}\left((T-t)^{H-1/2}+ \right.\\
    \left. \int_t^T (s-t)^{H-1/2}\left(s^{H-3/2}+s^{H-1/2}\right)ds\right)+C2^{-2nH} \\ \leq C(\epsilon_n+\delta_n)t^{1/2-H}\left((T-t)^{H-1/2}+ \int_t^T (s-t)^{H-1/2} s^{H-3/2} ds\right)+C2^{-2nH}\\\leq C(\epsilon_n+\delta_n)t^{1/2-H}\left((T-t)^{H-1/2}+ t^{2H-1}\int_1^\infty (s-1)^{H-1/2} s^{H-3/2} ds\right)+C2^{-2nH}\\ \leq C(\epsilon_n+\delta_n)\left(t^{1/2-H} (T-t)^{H-1/2}+t^{H-1/2}\right)+C2^{-2nH}.
    \end{gathered}\end{equation}
The above upper bound  \eqref{up-up} means that  $\int_0^T|K_T^H(f-f_n)(t)|^2dt\rightarrow 0$ as $n\rightarrow \infty,$ and in turn, it means    that $\E\left(\int_0^T f(t) dB^H(t)-\int_0^T f_n(t) dB^H(t)\right)^2\rightarrow 0$ as $n\rightarrow \infty.$

 It means that there exists a subsequence $n_k$ such that $\int_0^T f_{n_k}(t) dB^H(t)\rightarrow\int_0^T f(t) dB^H(t)$ a.s.  Without loss of generality denote this subsequence $n$.
 Now, for a step function $f_n$, $\int_0^T f_n(t) dB^H(t)=f_n(T)B^H(T)-\int_0^T B^H(t) df_n(t)$. Total boundedness of   variation
 together with uniform convergence of $f_n$ to $f$ allows to apply Helly-Bray theorem and conclude that for any $\omega \in \Omega,\; \mathbb{P}(\Omega)=1$, there exists a subsequence $n_k(\omega)$ such that $\int_0^T B^H(t) df_{n_k}(t)\rightarrow \int_0^T B^H(t) df(t).$ It means that with probability 1, $\int_0^T f(t) dB^H(t)=f(T)B^H(T)-\int_0^T B^H(t) df(t)$, and the lemma is proved.
\end{proof}
\begin{cor}\label{coroldiffer} Consider the function of the form $\varphi(t)=t^{1/2-H} f(t),$ where $f\in C^{(1)}([0,T])$. Then it is possible that $\varphi$ is  not in $C^{(1)}([0,T])$. However, it is bounded, continuous and of bounded variation. It is possible to choose the approximating sequence of functions $f_n=f_n(t)$ in such a way that both $f_n$ and $\varphi_n(t)=t^{H-1/2} f_n(t)$ be in $C^{(1)}([0,T])$, $f_n(t)\rightarrow f(t)$,  $f^\prime_n(t)\rightarrow f^\prime(t)$, $\varphi_n(t)\rightarrow \varphi(t)$ and $\varphi^\prime_n(t)\rightarrow \varphi^\prime(t)$ point-wise, $\varphi_n$ and $\varphi^\prime_n$ are totally bounded  and the total variation $V(f_n,[0,T])<\infty$. In this case, again by Helly-Bray theorem we   conclude that for any $\omega \in \Omega,\; \mathbb{P}(\Omega)=1$ there exists a subsequence $n_k(\omega)$ such that $\int_0^T B^H(t) d\varphi_{n_k}(t)\rightarrow \int_0^T B^H(t) d\varphi(t).$ Moreover, similarly to \eqref{operator int} and \eqref{operator-K}, but taking into account that $s^{H-1/2}(\varphi_{n}(s)-\varphi(s))= f(s)-f_n(s),$ we can conclude that
   \begin{equation*}\begin{gathered}K_T^H(\varphi-\varphi_n)(t) = C t^{1/2-H}\bigg( (T-t)^{H-1/2} (f(T)-f_n(T))\\+\int_t^T {(s-t)^{H-1/2}}  (f^\prime(s)-f^\prime_n(s)) ds,
     \end{gathered}\end{equation*}
     and therefore $\int_0^T|K_T^H(\varphi-\varphi_n)(t)|^2dt\rightarrow 0$ as $n\rightarrow \infty.$ It means that
     $\int_0^T \varphi(t) dB^H(t)=\varphi(T)B^H(T)-\int_0^T B^H(t) d\varphi(t).$

   \end{cor}
Now we formulate  Girsanov theorem for fractional Brownian motions.
\begin{theo} Let $B^H=\{B^H(t), t\geq 0\}$ be a fractional Brownian motion,  and let the function $\varphi\in L_1([0,T])$ satisfy $(K^H_{0}\varphi)(\cdot)\in L_2([0,T]).$ Then the process $\widetilde{B}^H=\{\widetilde{B}^H_t, t\in[0,T]\}$ defined by
    \begin{equation*}
        \widetilde{B}^H_t= {B}^H_t-\int_0^t\varphi(s)ds
    \end{equation*}
        is a fractional Brownian motion w.r.t. probability measure $\mathbb{Q}_H$ given by
    \begin{equation*}
        \frac{d\mathbb{Q}_H}{d\mathbb{P}}=\exp\left\{\int_0^T(K^{H,*}_{0}\varphi)(t)dB(t)-1/2\int_0^T\left((K^{H,*}_{0}\varphi)(t)\right)^2 dt\right\},
    \end{equation*}
        where  $B =\{B(t), t\geq 0\}$ is the underlying Wiener process.
        \end{theo}

\section{Lower and upper bounds for small ball probability}\label{sec: small-ball}

Let $B^H$ be an fBm with Hurst index $H \in (0,1/2)$ and $W$ be a Wiener process independent of $B^H$, both defined on our filtered probability space $(\Omega,\F,\{\F_t\}_{t\geq0},\mathbb{P})$.

We consider a mixed Gaussian process composed of the fBm $B^H$ and the Wiener process $W$. Introduce the notation
  $A_{T,g,\varepsilon}=\{\myset\}$. Our goal is to study the asymptotics of the probability
\begin{equation}
	\mathbb{P}^\varepsilon_g \coloneqq \mathbb{P}(A_{T,g,\varepsilon})
	\label{eq:prob}
\end{equation}
as $\epsilon \rightarrow 0$. The class of functions  $g$ is described as follows:  $g \in AC([0,T])$, i.e., $g$ admits a representation $g (t) = \int_0^t g^\prime (s)~ds$, and we   assume additionally that    $g^\prime\in L_{2}([0,T])$. Also, let $f:[0,T] \rightarrow \mathbb{R}_+$ be any positive measurable function.   It should  be mentioned immediately that in the case when $g\equiv0$, we can apply tools of the "small ball" theory for Gaussian processes to study the probability
\begin{equation*}
	\mathbb{P}^\varepsilon_0 \coloneqq \mathbb{P}(A_{T,0,\varepsilon}):=\mathbb{P}(|B^H(t)+W(t)|\leq \varepsilon f(t), t\in[0,T]),
	 \end{equation*} (see, e.g., a detailed survey \cite{li-shao}). We shall discuss this question in detail in Section \ref{section-small}. Therefore our main goal is to ``annihilate'' the trend $g$.

\subsection{Girsanov theorem and trend ``annihilation''}
In order to remove  the trend, we split $g$ into two differentiable functions $g_W(t) = \int_0^t g_W^\prime (s)~ds$ and $g_B(t) = \int_0^t g_B^\prime (s)~ds$, satisfying
\begin{equation*}
	g(t) = g_W(t)+g_B(t),
\end{equation*}
The classes for $g_B^\prime$ and $g_W^\prime$ must   be chosen in order to be able to apply the standard Girsanov theorem to the trend   $g_W(t) =\int_0^t g_W^\prime (s)~ds$ and the Girsanov theorem for fBm to the trend $g_B(t) =\int_0^t g_B^\prime (s)~ds$. According to Remark \ref{remark1}, we need $  g_B^\prime\in  L_{2}([0,T])$, and so we introduce the following notation:
  the expansion
  \begin{equation}\label{expansion}g=g_W+g_ B\end{equation}
   is \textbf{suitable} if
 $g_W, g_B \in AC([0,T])$ and $g^\prime_W, g^\prime_B \in L_{2}([0,T])$.
Now we ``share'' the trend between the fBm and the Wiener process according to some suitable expansion of $g$, so that
\begin{equation}\label{expansion-1}
	\BH(t)+W(t)+g(t) = \left(\BH(t) + g_B(t)\right) + \left(W(t) + g_W(t)\right).
\end{equation}
Since the processes $B^H$ and   $W$ are independent, we apply \eqref{expansion-1} and Theorem \ref{theo-girs}, $(iii)$, to present  the following obvious statement.
\begin{lemma}\label{lemma-1} Let $g \in AC([0,T])$,   $g (t) = \int_0^t g^\prime (s)~ds$ with $g^\prime\in L_{2}([0,T])$.   For any suitable expansion \eqref{expansion},
define the probability measure $\mathbb{Q}$  by
\begin{align*}\frac{d\mathbb{Q}}{d\mathbb{ P}}
&=\frac{d\mathbb{Q}_{B^H}}{d\mathbb{P}}\cdot\frac{d\mathbb{Q}_W}{d\mathbb{P}}\\
&:=\exp\left\{-\int_0^T(K^{H,*}_{0}g_B^\prime)(t)dB(t)
-1/2\int_0^T((K^{H,*}_{0}g_B^\prime)(t))^2 dt\right\}\\ &\qquad\times\exp\left\{-\int_0^Tg'_W(t)dW(t)-\int_0^T(g'_W(t))^2d t\right\},\end{align*}
where $B$ is the underlying Wiener process for $B^H$.
Then $\widetilde{B}^H(t):={B}^H(t)+\int_0^tg^\prime_B(s)ds$ is an fBm, $\widetilde{W}(t):=W(t)+\int_0^tg^\prime_W(s)ds$ is an independent Wiener process and, consequently, the stochastic process $\BH(t)+W(t)+g(t)$ equals  a sum $\widetilde{B}^H(t)+\widetilde{W}(t)$ of an independent fBm and a Wiener process, w.r.t. the measure $\mathbb{Q}$.
\end{lemma}
With the help of Lemma \ref{lemma-1} we can rewrite the probability under consideration as follows.
\begin{lemma}\label{lemma-2}  Probability $\mathbb{P}^\varepsilon_g$ can be written as
\begin{equation}\begin{gathered}\label{eq:prob_cm}
	 \mathbb{P}^\varepsilon_g=
  \mathbb{E}\bigg(\1_{A_{T,0,\varepsilon}}
	 \exp\bigg\{\int_0^T g_W^\prime(t) dW(t)  - \frac12 \int_0^T (g_W^\prime(t))^2 dt\\  + \int_0^T h(t) dB(t) - \frac12 \int_0^T (h(t))^2 dt\bigg\}\bigg)  = \exp\left\{- \frac12 \int_0^T (g_W^\prime(t))^2 + (h(t))^2 dt\right\}\\
		  \times\left\{\mathbb{E}\left[\1_{ A_{T,0,\varepsilon}}
		\left(\exp\left\{\int_0^T g_W^\prime(t) dW(t) + \int_0^T h(t) dB(t) \right\}-1\right)\right]
		 + \mathbb{P}^\varepsilon_0\right\}.
\end{gathered}\end{equation}	
 Here $h(t) = (K^{H,*}_{0}g_B^\prime)(t)\in L^2([0,T])$ and   $B$ is the underlying Wiener process for $B^H$ consequently a Wiener process independent of $W$.\end{lemma}
\begin{proof} Denote $\widetilde{W}(t)=W(t)+\int_0^tg^\prime_W(s)ds$ and $\widetilde{B}^H_t={B}^H_t+\int_0^tg^\prime_B(s)ds$. Using Lemma \ref{lemma-1}, we can write
 \begin{equation*}\begin{gathered}
	 \mathbb{P}^\varepsilon_g:=\mathbb{P}(A_{T,g,\varepsilon})  = \mathbb{E}_{\mathbb{Q}}\left(\1_{A_{T,g,\varepsilon}}\frac{d\mathbb{P}}{d\mathbb{ Q}}\right)
 \\=\mathbb{E}_{\mathbb{Q}}\left(\1_{\{|\widetilde{B}^H(t)+\widetilde{W}(t)| \leq \epsilon f(t), ~0 \leq t \leq T\}}\exp\left\{\int_0^T g_W^\prime(t) dW(t)  + \frac12 \int_0^T (g_W^\prime(t))^2 dt\right.\right.\\
	 \left.\left. + \int_0^T h(t) dB(t) + \frac12 \int_0^T (h(t))^2 dt\right\}\right)\\
	 =\mathbb{E}_{\mathbb{Q}}\left(\1_{\{|\widetilde{B}^H(t)+\widetilde{W}(t)| \leq \epsilon f(t), ~0 \leq t \leq T\}}\exp\left\{\int_0^T g_W^\prime(t) d\widetilde{W}(t)  - \frac12 \int_0^T (g_W^\prime(t))^2 dt \right.\right. \\
	 \left.\left. + \int_0^T h(t) d\widetilde{B}(t)  - \frac12 \int_0^T (h(t))^2 dt\right\}\right) \\
  =\mathbb{E}\left(\1_{A_{T,0,\varepsilon}}
  \exp\left\{\int_0^T g_W^\prime(t) dW(t)  - \frac12 \int_0^T (g_W^\prime(t))^2 dt  \right.\right.\\
  \left.\left. + \int_0^T h(t) dB(t) - \frac12 \int_0^T (h(t))^2 dt\right\}\right)\\ = \exp\left\{- \frac12 \int_0^T (g_W^\prime(t))^2 + (h(t))^2 dt\right\}\\
		 \times\Big\{\mathbb{E}\left[\1_{\{A_{T,0,\varepsilon}\}}
		\left(\exp\left\{\int_0^T g_W^\prime(t) dW(t) + \int_0^T h(t) dB(t) \right\}-1\right)\right]
	 + \mathbb{P}\left(A_{T,0,\varepsilon}\right)
		\Big\}.
\end{gathered}\end{equation*}
\end{proof}

\subsection{Lower bound as a function of  sharing   the trend between $W$ and $B^H$}
In this section, our goal is to apply  Lemma \ref{lemma-2} in order to get a lower bound for  $\mathbb{P}^\varepsilon_g$  that we then maximize among suitable expansions \eqref{expansion} of $g$. We recall that the expansion \eqref{expansion} is suitable if
 $g_W, g_B \in AC([0,T])$ and $g^\prime_W, g^\prime_B \in L_{2}([0,T])$.
\begin{lemma}\label{lem:lower_bound}
	Let $B^H$ be an fBm with Hurst index $H \in (0,1/2)$,  $W$ be a Wiener process independent of $B^H$ and $g \in AC([0,T])$,   with $g^\prime\in L_{2}([0,T])$.  Then  for any $\epsilon >0$ and any suitable expansion \eqref{expansion}  we have that
	\begin{align}\label{lower}
		&\mathbb{P}^\varepsilon_g=\mathbb{P}(A_{T,g,\varepsilon})  \geq \exp\left\{- \frac12 \int_0^T \left((g_W^\prime(t))^2 + (h(t))^2\right) dt\right\} \mathbb{P}\left(A_{T,0,\varepsilon}\right)\nonumber\\
		&\qquad =\exp\left\{- \frac12 \int_0^T \left((g_W^\prime(t))^2 + (h(t))^2\right) dt\right\} \mathbb{P}^\varepsilon_0,
	\end{align}
 where $h(t) = (K^{H,*}_{0}g_B^\prime)(t)$.
\end{lemma}

\begin{proof}
	To obtain a lower bound \eqref{lower} for \eqref{eq:prob}, we take the right-hand side of relations  \eqref{eq:prob_cm}
	 and  apply  inequality $e^x -1 \geq x$ to get the inequality
	\begin{align*}
		\mathbb{P}^\varepsilon _g &\geq \exp\left\{- \frac12 \int_0^T \left((g_W^\prime(t))^2 + (h(t))^2\right) dt\right\}\\
		&\qquad \times\left\{\mathbb{E}\left[\1_{ A_{T,0,\varepsilon} } \left(\int_0^T g_W^\prime(t)~ dW(t) + \int_0^T h(t)~ dB(t)\right)\right]
		  + \mathbb{P}\left(A_{T,0,\varepsilon}\right)
		\right\}\\
		& = \exp\left\{- \frac12 \int_0^T \left((g_W^\prime(t))^2 + (h(t))^2\right) dt\right\}\mathbb{ P}\left(A_{T,0,\varepsilon}\right).
	\end{align*}
	To get the last equality, we follow the idea from \cite{novikov} to use the fact that
	\begin{equation*}
		\mathbb{E}\left[\1_{A_{T,0,\varepsilon}} \left(\int_0^T g_W^\prime(t)~ dW(t) + \int_0^T h(t)~ dB(t)\right)\right] = 0.
	\end{equation*}
	Indeed, the expectation is taken on a centrally symmetric set and $-W(t)$ and $-B(t)$ are also Brownian motions.
\end{proof}

\subsubsection{Equation for the maximizer of the lower bound in terms of fractional integrals} In order to tighten the lower bound, we search for the function $g_B  = g -g_W $ that maximizes $\exp\left\{- \frac12 \int_0^T \left((g_W^\prime(t))^2 + (h(t))^2\right) dt\right\}$. Thus, we want to solve the following minimization problem among suitable expansions \eqref{expansion} of $g$:
\begin{equation}
	\min_{g^\prime_B \in L_2([0,T])} \int_0^T \left((g_W^\prime(t))^2 + (h(t))^2\right) dt.
	\label{eq:minimization}
\end{equation}
  That is, we consider  $g^\prime_W(t) = g^\prime(t)-g^\prime_B(t)$,  $h(t) = (K^{H,*}_{0}g_B^\prime)(t)$, and all components $g^\prime, g^\prime_W, g^\prime_B, h$ should be in $L^2([0,T])$. The following proposition presents a necessary condition for $g_B $ to be the minimizer.

\begin{lemma}\label{prop:necessary_min}
	Suppose that $g_B$ minimizes \eqref{eq:minimization} among suitable expansions \eqref{expansion} of $g$. Then $g^\prime_B$ satisfies the following equation
	\begin{equation}
		C_1^{-2} t^{1/2-H}\left(I_{T^-}^{1/2-H} \left(\cdot^{2H-1}  I_{0}^{1/2-H} (\cdot^{1/2-H}g^\prime_B)\right)\right)(t)  +   g^\prime_B(t) = g^\prime(t),
		\label{eq:min_condition2}
	\end{equation}
	for all $t \in [0,T]$.
\end{lemma}
\begin{remark} Obviously, equation \eqref{eq:min_condition2} can be rewritten in the following equivalent form
\begin{equation}\label{eq:min_condition2-1}K_T^{H,*}\left(K_0^{H,*}(g^\prime_B)\right)(t)+g^\prime_B(t)=g^\prime(t),\; t\in[0,T].\end{equation}

\end{remark}

\subsubsection{Equation for the maximizer as a Fredholm integral  equation and its properties}
Let us present equation  \eqref{eq:min_condition2} in a more appropriate form. In this mindset, denote ${C_2=\frac{C_1^{-2}}{\left(\Gamma(1/2-H)\right)^2}}$ and let, for brevity, $x=g_B^\prime$. Then, using  Definition \ref{def:RLint} of the fractional integrals,   the first term on the left-hand side of \eqref{eq:min_condition2} can be rewritten as
	\begin{align}
		& C_2
		t^{1/2-H}\int_t^T (u-t)^{-1/2-H} u^{2H-1}
		\int_0^u (u-z)^{-1/2-H} z^{1/2-H}x(z) ~dz ~du\nonumber\\
		\begin{split}
		 & \qquad = C_2
		t^{1/2-H}\left(\int_0^t x(z)z^{1/2-H} \int_t^T (u-t)^{-1/2-H} u^{2H-1} (u-z)^{-1/2-H}~du~dz\right.\\
	 &\qquad\left.  +t^{1/2-H}\int_t^T x(z)z^{1/2-H} \int_z^T (u-t)^{-1/2-H} u^{2H-1} (u-z)^{-1/2-H}~du~dz\right).
	 \label{eq-int-1}
	 \end{split}
	\end{align}
Next, taking into account representation \eqref{eq-int-1}, we can define the integral kernel $\kappa(z,t)$ as follows
	\begin{align}
	\begin{split}
		\kappa(z,t)&=(tz)^{1/2-H}\int_t^T (u-t)^{-1/2-H} u^{2H-1} (u-z)^{-1/2-H}~du\1_{\{0 \leq z < t\leq T\}}\\
		& \qquad+ (tz)^{1/2-H}\int_z^T (u-t)^{-1/2-H} u^{2H-1} (u-z)^{-1/2-H}~du\1_{\{0\leq t <z \leq T\}},
			 \label{kernel}
	\end{split}
	\end{align}
With the help of this kernel,  \eqref{eq:min_condition2} is reduced to a Fredholm integral equation
	\begin{equation}\label{Fredholm}
	x(t) + C_2 \int_0^T  {\kappa}(z,t) x(z) ~ dz
	= g^\prime(t), \;t\in[0,T].
	\end{equation}
Now our goal is to study the properties of the kernel $\kappa$ and to establish existence and uniqueness of the solution of Equation 	 \eqref{Fredholm} in the different classes of functions, depending on the value of $H$.
\begin{lemma}\label{properties}
	The kernel $\kappa$ described in \eqref{kernel} has the following properties
	\begin{itemize}
		\item[$(i)$] $\kappa$ is non-negative and symmetric.
		\item[$(ii)$] $\kappa$ is a polar kernel, more precisely, $\kappa (z,t)=\frac{\kappa_0 (z,t)}{|t-z|^{2H}},$ where the function $\kappa_0\in C([0,T]^2)$.
		\item[$(iii)$] \begin{itemize}
			\item[$(a)$] There exists a constant $C>0$ such that for any $1\leq r<\frac{1  }{2H}$
			$$\sup_{t\in[0,T]}\int_0^T \kappa(z,t)^rdz\leq C.$$
			\item[$(b)$] For any $1\leq r<\frac{1  }{2H}$, $\kappa$ belongs to  $L_r([0,T]^2)$.
		\end{itemize}
		\item[$(iv)$] $\kappa$ is a non-negative  definite kernel; more precisely, for any $x\in L_2([0,T])$
		the value $$\int_0^T\int_0^T {\kappa}(z,t)x(z)x(t)~dz~dt$$ is well-defined and nonnegative.
	\end{itemize}
\end{lemma}
\begin{remark}\label{main rem} \begin{itemize}
\item[$(i)$] Taking into account  $(iii)(a)$, we can apply   Theorem \ref{kant-aki-1} from the Appendix with any $r=\sigma<1/2H$ and conclude that the integral operator $Ax(s)=C_2\int_0^T\kappa(s,t)x(t)dt$ is a linear continuous operator from $L_p([0,T])$ into $L_q([0,T])$ for any $q\geq p>1$, $q\geq r$, $r\geq \frac{pq}{pq-q+p}$. In particular, we can let $p=q$. In this case $\frac{pq}{pq-q+p}=1<\frac{1}{2H}$. Therefore  we can put $r=1$ and then for any $ p>1$ the integral operator $A$ is a linear continuous operator from $L_p([0,T])$ into $L_p([0,T])$.
		\item[$(ii)$] Furthermore, let us apply Theorem \ref{kant-aki-2} with any $1<r=\sigma<1/2H$ and try to consider $p=q$. In this case, the inequality $\left(1-\frac{\sigma}{q}\right)\frac{p}{p-1}< r$ becomes $\left(1-\frac{r}{p}\right)\frac{p}{p-1}=\frac{p-r}{p-1}< r$, which is obviously true. Therefore, if we take any  $p=q>1$, and choose $ r=\sigma<(p\wedge \frac1{2H})$, then the conditions of Theorem \ref{kant-aki-2} will hold and we get that the operator $A$ will be compact from $L_p([0,T])$ into $L_p([0,T])$. Of course, it follows from the symmetry of the kernel $\kappa$ that the operator $A$ is self-adjoint, therefore its adjoint operator $A^*=A$ is compact from $L_p([0,T])$ into $L_p([0,T])$ for any $p>1$ as well.
		
	\end{itemize}
\end{remark}

\begin{cor} Summarizing the properties of operator $A$ from Lemma \ref{properties}, we see that $L_2([0,T])$ is the natural class in which $A$ is compact and non-negative definite. \end{cor}
The next principal result is an immediate consequence of Lemma \ref{properties}, Remark \ref{main rem} and the Fredholm alternative.
\begin{theo}\label{prop:solution3.5}
	Let $H<1/2$, $g^\prime\in L_2([0,T])$. Then there exists a unique function $x=x(t), t \in [0,T], x\in L_2([0,T])$ that satisfies \eqref {Fredholm},   for all $t \in [0,T]$. Consequently, the same is true for equivalent equations   \eqref{eq:min_condition2} and \eqref{eq:min_condition2-1}.
\end{theo}

It follows that  there exists a candidate minimizer $g^\prime_B(t)$ to \eqref{eq:minimization} which equals  the unique solution of
\eqref {Fredholm}. Moreover, it follows from Remark \ref{remark1} that this function creates an admissible trend in the sense that for $\varphi=g^\prime_B$, Girsanov theorem (Theorem \ref{theo-girs}, $(iii)$) holds.

Next, we verify that the candidate minimizer $g^\prime_B(t)$ satisfying \eqref{eq:min_condition2} is in fact the solution to \eqref{eq:minimization}.

\begin{theo}\label{prop:sufficient}
	Let   $x  $ be the unique $L_2([0,T])$-solution of  \eqref{eq:min_condition2}. Then $x$ is the unique solution to \eqref{eq:minimization}.
\end{theo}

\begin{proof}
	From Theorem \ref{thm:thm2almeida} (see Appendix \ref{app:A}), the candidate minimizer obtained in Theorem  \ref{prop:solution3.5} minimizes \eqref{eq:minimization} if $L(t,x,y)$, given by \eqref{eq:Lfct} satisfies
	\begin{align*}
	L(t,x+x_1,y+y_1) - L(t,x,y) \geq \partial_1 L(t,x,y)x_1 + \partial_1 L(t,x,y)y_1
	\end{align*}
	for all $(t,x,y)$, $(t,x+x_1,y+y_1)$ in $[0,T]\times \R^2$, where $\partial_1$ ($\partial_2$) stand, respectively, for the differentiation in $x$ (in $y$).
	Using \eqref{eq:Lfct}, the condition is equivalent to
	\begin{equation*}
	x_1^2 + 2C^{-2}_1 t^{2H-1} y_1^2 \geq 0,
	\end{equation*}
	which is satisfied for any $(x_1,y_1) \in \R^2$.
\end{proof}
Finally, taking into account Lemma \ref{lem:lower_bound} and Theorems \ref{prop:solution3.5} and \ref{prop:sufficient}, we get the following main result  concerning  the lower bound.
\begin{theo}\label{lower-all} Let $B^H$ be an fBm with Hurst index $H \in (0,1/2)$,  $W$ be a Wiener process independent of $B^H$ and $g \in AC([0,T])$,   with $g^\prime\in L_{2}([0,T])$.  Then  for any $\epsilon >0$ and any suitable expansion \eqref{expansion}  we have that
	\begin{align}\label{lower-max}
		& \mathbb{P}(A_{T,g,\varepsilon})  \geq \exp\left\{- \frac12 \int_0^T \left((g_W^\prime(t))^2 + (h(t))^2\right) dt\right\} \mathbb{P}\left(A_{T,0,\varepsilon}\right),
	\end{align}
 where $h(t) = (K^{H,*}_{0}g_B^\prime)(t)$, and $g_B^\prime(t)$  is the unique $L^2([0,T])$-solution of Equation \eqref{Fredholm}. Lower bound \eqref{lower-max} is optimal in the sense that $$\frac12 \int_0^T \left((\widetilde{g}_W^\prime(t))^2 + (\widetilde{h}(t))^2\right) dt\geq \frac12 \int_0^T \left((g_W^\prime(t))^2 + (h(t))^2\right) dt$$ for any other suitable expansion $g^\prime=\widetilde{g}_W^\prime+\widetilde{g}_B^\prime.$
\end{theo}

\subsection{Upper bound}

In this section, we use Lemma \ref{lemma-2} to obtain an upper bound for the probability $\mathbb{P}^\varepsilon_g=\mathbb{P}(A_{T,g,\varepsilon}),$  in terms of a particular suitable expansion \eqref{expansion} of $g$. We recall that an expansion \eqref{expansion} is called suitable if $g_W, g_B \in AC([0,T])$ and $g^\prime_W,g^\prime_B \in L_2([0,T])$.

The next auxiliary  result will be used for obtaining an upper bound.

\begin{lemma}\label{prop:upper_bound}
	Let $B^H$ be an fBm with Hurst index $H \in (0,1/2)$,  $W$ be a Wiener process independent of $B^H$ and $g \in AC([0,T])$,   with $g^\prime\in L_{2}([0,T])$. Let also $h(t) = (K^{H,*}_{0} g^\prime_B)(t)$. Then there exists a sequence of real numbers $C_n,n\geq 1$, and vanishing sequence $c_n,n\geq 1$ such that $c_n\rightarrow 0$ as $n\rightarrow 0$ such that for any $\epsilon >0$, we have the sequence of inequalities
	\begin{equation}\begin{gathered}\label{auxil-upper}
		 \mathbb{P}(A_{T,g,\varepsilon})
		   \leq \exp\left\{- \frac12 \int_0^T \left((g_W^\prime(t))^2 + (h(t))^2\right) dt \right\}\mathbb{P}\left(A_{T,0,\varepsilon}\right)\\+\mathbb{P} (A_{T,0,\varepsilon})(\exp\{\varepsilon C_n+c_n\}-1)
,\end{gathered}\end{equation}
	where   $g_W(t) = g(t) - g_B(t)$, and   $g^\prime_B$ is the unique solution of Equation \eqref{eq:min_condition2-1}
	\begin{equation*}
		g^\prime(t)-g^\prime_B(t) = (K^{H,*}_{T} h)(t)=\left(K^{H,*}_{T}\left(K^{H,*}_{0} g^\prime_B\right)\right)(t).
	\end{equation*}
\end{lemma}
\begin{proof}
	Let $g_B\in L^2([0,T])$ satisfy Equation \eqref{eq:min_condition2-1} which can be rewritten as
	\begin{equation}\label{once again}
		g^\prime_W(t) = g^\prime(t)-g^\prime_B(t) =(K^{H,*}_{T} h)(t)= K^{H,*}_{T}\left(K^{H,*}_{0} g^\prime_B\right)(t).
		 \end{equation}
	By Theorem \ref{prop:solution3.5}, Equation \eqref{once again} has a unique solution  $g^\prime_B\in L^2([0,T])$. Furthermore, the expansion ${g_B = \int_0^\cdot g^\prime_B(t) dt}$, ${g_W = g - g_B}$ is suitable.
	
	Now, according to Remark \ref{remark1}, $h  = (K^{H,*}_{0}g_B^\prime) $ is in $L_2([0,T])$.  Using the definition of the integral w.r.t. an fBm from subsection \ref{subsec-2.3} with $f= K^{H,*}_T h\in K_H([0, T])$, we obtain
	\begin{equation}\label{eq:h-Weiner-int}
		\int_0^T (K^{H,*}_T h)(t)~ dB^H(t)= \int_0^T h(t) ~dB(t),
	\end{equation}
 where $B$ is the underlying Brownian motion of $B^H$.
From \eqref{once again} and \eqref{eq:h-Weiner-int}  we get that
	\begin{align*}
		\int_0^T g_W^\prime(t)~ dW(t) + \int_0^T h(t)~ dB(t)
		& = \int_0^T g_W^\prime(t)~ dW(t) + \int_0^T (K^{H,*}_{T} h)(t)~ d\BH(t)\\
		 &=  \int_0^T g_W^\prime(t)~ dW(t) + \int_0^T g^\prime_W(t)~ d\BH(t).
	\end{align*}
  Now the problem is that, generally speaking,  we can not  consider the sum of integrals $$\int_0^T g_W^\prime(t)~ dW(t) + \int_0^T g^\prime_W(t)~ d\BH(t)$$ as one integral w.r.t. to the mixture of processes because the integrals   are defined in different ways. However, we can apply Lemma \ref{boudvar}.
	To begin, choose a sequence of continuously differentiable functions $h_n$ satisfying two additional assumptions:
  \begin{itemize}
  \item [$(i)$] $\int_0^T |h(t)-h_n(t)|^2dt \rightarrow 0 \qquad \text{as }n \rightarrow \infty$;
  \item [$(ii)$] there exists a sequence of real numbers $\beta_n>0, \beta_n\rightarrow 0\; \text{as }n \rightarrow \infty$ such that $h_n(t)=0$ for $t\in [0, \beta_n]$.
      \end{itemize}
       Further,  create two sequences
\begin{equation*}\begin{gathered} g^\prime_{W,n}(t)=(K_T^{H,*}h_n)(t)=C_1t^{1/2-H}\int_{t}^T(s-t)^{-1/2-H}s^{H-1/2}h_n(s)ds\\
=C_1t^{1/2-H}\int_{t\vee\beta_n}^T(s-t)^{-1/2-H}s^{H-1/2}h_n(s)ds,
\end{gathered}\end{equation*}
and
\begin{equation*}\begin{gathered} g^\prime_{B,n}(t)=(K_0^{H}h_n)(t)= C_1 t^{H-1/2}\frac{d}{dt}\int_{0}^t(s-t)^{-1/2-H}s^{1/2-H}h_n(s)ds.
\end{gathered}\end{equation*}
Integral  $\int_{t\vee\beta_n}^T(s-t)^{-1/2-H}s^{H-1/2}h_n(s)ds$ is continuously differentiable   because   the  derivative of the integral  equals
$$\int_{t\vee\beta_n}^T(s-t)^{-1/2-H}(s^{H-1/2}h^\prime_n(s)+(H-1/2)s^{H-3/2}h_n(s))ds\in C([0,T]),$$
 and note that $g^\prime_{W,n}$ satisfies
	\begin{equation*}
	\int_0^T |K_T^H(g_W^\prime  - g^\prime_{W,n})(t)|^2 dt=\int_0^T |h(t)-h_n(t)|^2dt \rightarrow 0
	 \end{equation*}
  as $n \rightarrow \infty$.

	Taking into account   Corollary \ref{coroldiffer}, we conclude that   it is possible to  apply integration by parts to the integral $\int_0^T g^\prime_{W,n}dB^H(t)$,     for each $n > 0$,  and to obtain that
	\begin{align*}
		&\int_0^T g^\prime_{W,n}(t)~ dW(t) + \int_0^T g^\prime_{W,n}(t)~ d\BH(t)\\
		& \qquad = -\int_0^T W(t) d(g^\prime_{W,n}(t)) + W(T) g^\prime_{W,n}(t) - \int_0^T \BH(t) d(g^\prime_{W,n}(t)) + \BH(T) g^\prime_{W,n}(t)\\
		& \qquad = -\int_0^T (W(t) + \BH(t)) d(g^\prime_{W,n}(t)) + (W(T)+\BH(T)) g^\prime_{W,n}(T).
	\end{align*}
	
	  Therefore, on the set $\{|\BH(t)+W(t)|\leq \epsilon f(t),~0 \leq t \leq T\}$, we have
	\begin{equation}\begin{gathered}\label{eq:bound_gn}
		 \left|-\int_0^T (W(t) + \BH(t)) d(g^\prime_{W,n})(t) + (W(T)+\BH(T)) g^\prime_{W,n}(T)\right|\\
\leq  \int_0^T |W(t) + \BH(t)| d(|g^\prime_{W,n}|)(t) + |W(T)+\BH(T)| |g^\prime_{W,n}(T)| \\
\leq   \varepsilon \int_0^T f(t)  d(|g^\prime_{W,n}|)(t)+\varepsilon f(T) |g^\prime_{W,n}(T)|
	  = \epsilon C_n,
		\end{gathered}\end{equation}
	where $C_n $ is a constant not depending on $\epsilon$.
	Denote $$  \Delta_n(t) =  \int_0^t (g_W^\prime(s) - g^\prime_{W,n}(s)) ds+\int_0^t (g_B^\prime(s) - g^\prime_{B,n}(s)) ds, $$
and "distribute" the trend $  \Delta_n(t)$ among $W$ and $B^H $, accordingly to these two integrals, noticing that

 $$\int_0^t  K_0^{H,*}(g_B^\prime - g^\prime_{B,n})(s) dB(s)=\int_0^t (h-h_n)(s) dB(s).$$ Then, applying \eqref{eq:bound_gn} and Girsanov theorem for fBm (Lemma \ref{lemma-1}), we can rewrite and bound from above the terms from the right-hand side of \eqref{eq:prob_cm} as follows:
	\begin{equation*}\begin{gathered}
		 \mathbb{E}\left[\1_{A_{T,0,\epsilon}}
		\exp\left\{\int_0^T g_W^\prime(t) dW(t) + \int_0^T h(t) dB(t) \right\}\right]\\
 \leq\mathbb{E}\left[\1_{A_{T,0,\epsilon}}
		\exp\left\{ \epsilon C_n + \int_0^T (g_W^\prime(t) - g^\prime_{W,n}(t)) dW(t)
			+ \int_0^T (g_W^\prime(t) - g^\prime_{W,n}(t)) d\BH(t)\right\}\right]\\
		  = \exp\left\{\epsilon C_n + \frac12 \int_0^T (g_W^\prime(t) - g^\prime_{W,n}(t))^2 dt+\frac12 \int_0^T (K_T^H(g_W^\prime  - g^\prime_{W,n})(t))^2 dt\right\}\\ \times \mathbb{P}(|\BH(t)+W(t)+\Delta_n(t)|\leq \epsilon f(t), ~0 \leq t \leq T)\\=
		    \exp\left\{\epsilon C_n + \frac12 \int_0^T (g_W^\prime(t) - g^\prime_{W,n}(t))^2 dt+\frac12 \int_0^T   (h(t)  - h_n(t))^2 dt\right\}\\ \times \mathbb{P}(|\BH(t)+W(t)+\Delta_n(t)|\leq \epsilon f(t), ~0 \leq t \leq T)\\=
\exp\left\{\epsilon C_n +c_n \right\}  \mathbb{P}(|\BH(t)+W(t)+\Delta_n(t)|\leq \epsilon f(t), ~0 \leq t \leq T),
 \end{gathered}\end{equation*}
 where $c_n=\frac12 \int_0^T (g_W^\prime(t) - g^\prime_{W,n}(t))^2 dt+\frac12 \int_0^T   (h(t)  - h_n(t))^2 dt.$

	Next, using a version of the Anderson lemma (see, e.g, \cite{li-shao}, Theorem 3.1), we get that
	\begin{align*}
		\mathbb{P}(|\BH(t)+W(t)+\Delta_n(t)|\leq \epsilon f(t), ~0 \leq t \leq T)\leq
		  \mathbb{P}(A_{T,0,\epsilon}).
	 \end{align*}
In order to finish the proof, it is sufficient to recall  that $\int_0^T   (h(t)  - h_n(t))^2 dt\rightarrow 0$ consequently, applying the version of Hardy--Littlewood theorem from  Remark \ref{remark1}, we conclude that
$$\int_0^T (g_W^\prime(t) - g^\prime_{W,n}(t))^2 dt=\int_0^T   (K_T^{H,*}(h(t)  - h_n(t)))^2 dt\rightarrow 0$$ as $n\rightarrow \infty.$ Then the proof immediately follows from \eqref{eq:prob_cm}.
	\end{proof}
The next result is the main result of this section. It is an immediate corollary of the upper bound \eqref{auxil-upper}.
\begin{theo} In the conditions and terms of Lemma \ref{prop:upper_bound}
\begin{equation}\begin{gathered}\label{auxil-upper-2}
		 \lim_{\varepsilon\rightarrow 0} \mathbb{P}(A_{T,g,\varepsilon})
		   \leq \exp\left\{- \frac12 \int_0^T \left((g_W^\prime(t))^2 + (h(t))^2\right) dt \right\}\lim_{\varepsilon\rightarrow 0}\mathbb{P}\left(A_{T,0,\varepsilon}\right).
 \end{gathered}\end{equation}
\end{theo}
\section{Asymptotics of small deviations of mixed fractional Brownian motion}\label{section-small}

Although the exact value for the probability $\mathbb{P}\{|B^H(t) + W(t)| < \varepsilon f(t), 0 \leq t \leq T\}$ generally speaking, is unknown, its asymptotics can be established or estimated in some particular cases.

\begin{example}
	Let $f(t) = a$ be some constant (without loss of generality, we can assume that $a=1$). It was established in \cite{shao-1993} and \cite{monrad-rootzen}, see also \cite{Li-Linde} that
	\begin{equation*}
	-\log \mathbb{P}\left(\sup_{0 \leq t \leq T} |B^H(t)| \leq \epsilon\right) \sim \varepsilon^{-\frac{1}{H}}, \qquad \text{as } \varepsilon \rightarrow 0,
	\end{equation*}
	where $\sim$ means that their asymptotic behaviour is the same up to a constant multiplier.

\begin{lemma}\label{lem:asymptotic}
	The asymptotic behaviour of the small deviation for a mixed fractional Brownian motion, namely, of the value
$$-\log \mathbb{P}(\sup_{0 \leq t \leq T}|W(t) + B^H(t)| \leq \epsilon),$$ is the same as the one for a fractional Brownian motion $B^H$, if $H \in (0, \frac12)$, i.e.
	\begin{equation*}
	-\log \mathbb{P}\left(\sup_{0 \leq t \leq T} |B^H(t)+W(t)| \leq \epsilon\right) \sim \varepsilon^{-\frac{1}{H}}, \qquad \text{as } \varepsilon \rightarrow 0.
	\end{equation*}
\end{lemma}

\begin{proof}
	Consider $T=1$. On the one hand,
	\begin{equation*}
	\begin{gathered}
	\P\left(\sup_{0 \leq t \leq 1} |B^H(t) + W(t)| \leq \varepsilon \right)
	\geq \P\left(\sup_{0 \leq t \leq 1}|W(t)| \leq \frac\varepsilon 2\right)\cdot
	\P\left(\sup_{0 \leq t \leq 1}|B^H(t)| \leq \frac\varepsilon 2\right),
	\end{gathered}
	\end{equation*}
	and
	\begin{equation*}
	\begin{gathered}
	-\log\P\left(\sup_{0 \leq t \leq 1} |W(t) + B^H(t)| \leq \varepsilon\right)
	\leq -\log\P\left(\sup_{0 \leq t \leq 1} |W(t)| \leq \frac{\varepsilon}2\right)
	-\log\P\left(\sup_{0 \leq t \leq 1} |B^H(t)| \leq \frac{\varepsilon}2\right)\\
	\sim \varepsilon^{-2} + \varepsilon^{-\frac1H} \sim \varepsilon^{-\frac1H}
	\qquad\text{as } \varepsilon \rightarrow 0.
	\end{gathered}
	\end{equation*}
	On the other hand, we can apply Theorem 4.5 from \cite{li-shao}. According to this theorem, for any centered Gaussian process $X$ with stationary increments satisfying
	\begin{enumerate}[$(i)$]
		\item
		$X_0 = 0$,
		\item\label{cond:ii}
		$c_1 \sigma(h) \leq \sigma(2h) \leq c_2 \sigma(h)$, for $0 \leq h \leq \frac{1}{2}$ and some $1 < c_1 \leq c_2 < 2$,
		\item\label{cond:iii}
		$\sigma^2$ is concave on $(0,1)$,
	\end{enumerate}
	where $\sigma^2(|t-s|) = \E\left[|X_t-X_s|^2\right]$, we have that for any $0 < \varepsilon < 1$, there exists $0 < C < \infty$ such that
	\begin{equation*}
	\P(\sup_{0 \leq t \leq 1} |X_t| \leq \sigma(\varepsilon)) \leq e^{-\frac C \varepsilon}
	\end{equation*}
	In our case, $\sigma^2(|t-s|) = |t-s| + |t-s|^{2H}$, $0 < 2H <1$, and therefore conditions $(\ref{cond:ii})$ and $(\ref{cond:iii})$ are satisfied. It means that for any $0 < \varepsilon < 1$,
	\begin{equation*}
	\P\left(\sup_{0 \leq t \leq 1} |W(t) + B^H(t)| \leq \sqrt{\varepsilon + \varepsilon^{2H}}\right) \leq e^{-\frac C \varepsilon}.
	\end{equation*}
	If we replace $\varepsilon^H = \delta$, then $\varepsilon = \delta^{\frac 1H}$, and
\begin{equation*}
	  \P\left(\sup_{0 \leq t \leq 1} |W(t) + B^H(t)| \leq  \delta \right)\leq\P\left(\sup_{0 \leq t \leq 1} |W(t) + B^H(t)| \leq \sqrt{\delta^2+\delta^{\frac 1H}}\right)
 \leq e^{-\frac {C}{\delta^{\frac 1H}}},
	\end{equation*}
	whence
\begin{equation*}
	-\log \P\left(\sup_{0 \leq t \leq 1} |W(t) + B^H(t)| \leq \delta\right) \geq \frac C {\delta^{\frac 1H}},
	\end{equation*}
	and  the proof follows.
\end{proof}
\end{example}
\begin{remark}\label{rem-el-nouty} Generally speaking, the lower and upper functional classes  for Gaussian processes were the subject of detailed research. In particular,   the mixed fractional Brownian motion, integrated fractional Brownian motion and  sub-fractional Brownian motion   were studied from this point of view in \cite{el2003fractional}, \cite{el2004lower}, \cite{el2008fractional} and \cite{el2012lower}. However, for our simple result in Lemma \ref{lem:asymptotic}, it is more reasonable for the reader's convenience to give a direct and simple proof, which was done, rather than to provide some adaptation.
\end{remark}

\begin{example}
	Let $f$ be positive measurable function, $f \in AC([0,T])$ be separated from $0$, with $f'$ having bounded variation. Then it was established in \cite{novikov} that
	\begin{equation*}
	\P\left(|W(t)| \leq \frac \epsilon 2 f(t),~0 \leq t \leq T \right)
	= \left(\frac{f(T)}{f(0)}\right)^{1/2} \P\left(|W(t)| \leq \varepsilon,~0 \leq t \leq 4\int_0^T f^{-2}(u)du\right),
	\end{equation*}
	whence
	\begin{equation*}
	-\log \P\left(|W(t)| \leq \frac \varepsilon 2 f(t), 0 \leq t \leq T\right) \sim \varepsilon^{-2}.
	\end{equation*}
	Furthermore, if we denote $f^* = \max_{0 \leq t \leq T} f(t)$, then, on the one hand,
	\begin{equation*}
	P\left(|W(t) + B^H(t)| \leq \varepsilon f(t)\right) \geq P\left(|W(t)| \leq \frac{\varepsilon f(t)}2, 0 \leq t \leq T\right)\P\left(\sup_{0 \leq t \leq T}|B^H(t)| \leq \frac{\varepsilon f^*}{2}\right),
	\end{equation*}
	and on the other hand,
	\begin{equation*}
	\P\left(|W(t) + B^H(t)| \leq \varepsilon f(t), 0 \leq t \leq T\right)
	\leq \P\left(\sup_{0 \leq t \leq T} |W(t) + B^H(t)| \leq \varepsilon f^*\right).
	\end{equation*}
	Applying Lemma \ref{lem:asymptotic}, we get that
	\begin{equation*}
	- \log \P\left(|W(t) + B^H(t)| \leq \varepsilon f(t), 0 \leq t \leq T\right)
	\sim \varepsilon^{-\frac1H} \qquad \text{as } \varepsilon \rightarrow 0.
	\end{equation*}
\end{example}

\appendix

\section{Proofs and useful results}\label{app:A}
\subsection{Properties of integral operators}  Consider two classical theorems that give sufficient conditions of boundedness and compactness of the linear integral operators. The first theorem gives   sufficient conditions for the linear integral operator to be continuous.
\begin{theo}\label{kant-aki-1} [\cite{kant-ak}, Theorem 1, p. 324] Let $\kappa(s,t), s,t \in [0,T]$ be an integral kernel with the properties
	$$\sup_{s\in [0,T]}\int_0^T|\kappa(s,t)|^rdt\leq C, ~~~\sup_{t\in [0,T]}\int_0^T|\kappa(s,t)|^\sigma ds\leq C$$
	for some constant $C>0$ and for some $r,\sigma>0$. Then the integral operator $Ax(s)=\int_0^T\kappa(s,t)x(t)dt$ is a linear continuous operator from $L_p([0,T])$ into $L_q([0,T])$ for any $$q\geq p> 1,~ q\geq\sigma, ~\left(1-\frac{\sigma}{q}\right)\frac{p}{p-1}\leq r.$$
\end{theo}
If we strengthen the conditions of Theorem \ref{kant-aki-1}, we get a compact integral operator.
\begin{theo}\label{kant-aki-2} [\cite{kant-ak}, Theorem 3, p. 326] Let $\kappa(s,t), s,t \in [0,T]$ be an integral kernel with the properties
	$$\sup_{s\in [0,T]}\int_0^T|\kappa(s,t)|^rdt\leq C, ~~~\sup_{t\in [0,T]}\int_0^T|\kappa(s,t)|^\sigma ds\leq C$$
	for some constant $C>0$ and for some $r,\sigma>0$. Then the integral operator $Ax(s)=\int_0^T\kappa(s,t)x(t)dt$ is a linear compact operator from $L_p([0,T])$ into $L_q([0,T])$ for any $$q\geq p> 1,~ q>\sigma, ~\left(1-\frac{\sigma}{q}\right)\frac{p}{p-1}< r.$$
	
\end{theo}
\subsection{Equations for the minimizer of the integral   involving both function and its fractional integral}\label{app:A2}
Here we adapt Theorem 3.1 and Theorem 6.2 presented in \cite{almeida2009calculus} that are needed to solve minimization problem \eqref{eq:minimization}. The results we present here are slightly less general however suitable for our  context.
Both theorems are concerned with the following optimization problem:
\begin{equation}
\min_{\varphi\in L^2([0,T])} \int_0^T L(t,  \varphi(t), (I^{1-\alpha}_{0}\varphi)(t)) dt,
\label{eq:Lfct_appendix}
\end{equation}
where $L=L(t,x,y) \in C^1([0,T]\times \R^2; \R)$, ${t \rightarrow \partial_2 L(t,  \varphi(t), (I^{1-\alpha}_{0}\varphi)(t))}$ admits continuous left-side Riemann-Liouville fractional integral of order $(1-\alpha)$, $\alpha\in(0,1)$.

Here and below $\partial_1$ and $\partial_2$ denote the differentiation in $x$ and $y$, respectively.

The first theorem provides a necessary condition for $y$ to be a minimizer of \eqref{eq:Lfct_appendix}.

\begin{theo} \label{thm:thm1almeida}
	Let $y$ be a local minimizer of \eqref{eq:Lfct_appendix}. Then, $\varphi$ satisfies the fractional Euler-Lagrange equation
	\begin{equation}
	\left(I^{1-\alpha}_{T^-} \partial_2 L(u, \varphi (u), (I^{1-\alpha}_{0}\varphi)(u)\right)(t)
	+   \partial_1 L\left(u,  \varphi(u), (I^{1-\alpha}_{0}\varphi)(u)\right)(t)
	=0
	\label{eq:mincondition_app}
	\end{equation}
	for all $t \in [0,T]$.
\end{theo}	

This next theorem gives a sufficient condition for the candidate minimizer $y$ from Theorem \ref{thm:thm1almeida} to be a minimizer of \eqref{eq:Lfct_appendix}.

\begin{theo} \label{thm:thm2almeida}
	Let $L(t,x,y)$ satisfy
	\begin{equation*}
	L(t,x+x_1,y+y_1)-L(t,x,y) \geq \partial_1 L(t,x,y)x_1+\partial_2 L(t,x,y)y_1
	\end{equation*}
	for all $(t,x,y)$, $(t,x+x_1,y+y_1) \in [0,T]\times \R^2$. If $\varphi_0$ curve satisfies \eqref{eq:mincondition_app}, then $\varphi_0$ minimizes \eqref{eq:Lfct_appendix}.
\end{theo}
\subsection{Proofs of  the statements concerning integral operators and equations}
\textbf{Proof of Lemma \ref{prop:necessary_min}.}
 First note that using Definition \ref{def:RLint}, the integrand $(g_W^\prime(t))^2 + (h(t))^2$ in \eqref{eq:minimization} can be rewritten as
	\begin{align*}
		&(g_W^\prime(t))^2 + (h(t))^2\\
		&\qquad = (g^\prime(t))^2 - 2 g^\prime(t) g^\prime_B(t) + (g^\prime_B(t))^2 + C^{-2}_1 t^{2H-1}\left[\left(I_{0}^{1/2-H} \left(\cdot^{1/2-H} g^\prime_B\right)\right)(t)\right]^2.
	 \end{align*}
	The minimization \eqref{eq:minimization} can be described as
	\begin{equation*}
		\min_{x\in \in L_2([0,T])} \int_0^T L\left(t, x(t), I_{0}^{1/2-H}\left(\cdot^{1/2-H}x(\cdot)\right)\right) ~dt,
		 \end{equation*}
	where
	\begin{equation}
	L(t,x,y) =  - 2x g(t) + x^2 + C^{-2}_1 t^{2H-1}y^2.
	\label{eq:Lfct}
	\end{equation}

	Now we can apply standard minimization procedure, described, e.g., in \cite{almeida2009calculus} in a somewhat different, but in a sense, even more intricate situation. More precisely, let $x_0$ be a minimizing function. Consider the disturbed function $x_\varepsilon(t)= x_0(t)+\varepsilon \eta(t)$. The derivative of $\int_0^T L(t, x_\varepsilon(t), I_{0}^{1/2-H}x_\varepsilon(t)) ~dt$ in $\varepsilon$ at the point $\varepsilon=0$ should be zero.   However, it follows from the linearity of fractional integral that the derivative equals, up to a constant multiplier,

\begin{align}\label{many-int-1}&\int_0^T\bigg( - g(t)\eta(t)+x_0(t)\eta(t)\nonumber\\
&+C^{-2}_1 t^{2H-1}\left(I_{0}^{1/2-H}\left(\cdot^{1/2-H}x_0\right)\right) \big(t\big) \left(I_{0}^{1/2-H}\left(\cdot^{1/2-H}
\eta\right)\right)\big(t\big) \bigg)dt.
\end{align}
Applying   integration	by parts for fractional integral from \cite{samko1993fractional} we get that
\begin{align}
\label{many-int-2}
&\int_0^T	 t^{2H-1}\left(I_{0}^{1/2-H}\left(\cdot^{1/2-H}x_0\right)\right) \left(t\right)  \left(I_{0}^{1/2-H}\left(\cdot^{1/2-H}
\eta\right)\right)\big(t\big)  dt\nonumber\\
&=\int_0^T	 I_{T^-}^{1/2-H}\left(\cdot^{2H-1}\left(I_{0}^{1/2-H}\left(\cdot^{1/2-H}x_0\right)\right) \right) \big(t\big)  t^{1/2-H}\eta(t) dt,
\end{align}

Substituting \eqref{many-int-2} into \eqref{many-int-1} and taking into account that $\eta$ can be any bounded function, we get \eqref{eq:min_condition2}.
\hfill$\Box$

\textbf{Proof of Lemma \ref{properties}.}
 \begin{itemize}
		\item[$(i)$] It immediately follows from the representation  \eqref{kernel}.
		\item[$(ii)$] In order to establish this,  we note that $\kappa$ is nonnegative, and so it is sufficient to bound $\kappa$ from above.   First focus on $0 \leq z < t$. Writing $ {u}-t = \left(t-  z\right)x$,  for $0 \leq z < t$, we have
		\begin{align}
		\kappa(z,t)   &= (tz)^{1/2-H}\int_t^T (u-t)^{-1/2-H} u^{2H-1} (u-z)^{-1/2-H}~du \nonumber\\
		&= t^{2H-1}(tz)^{1/2-H} \left(t-z\right)^{-2H} \times\nonumber\\
		&\qquad \int_0^{\frac{T-t}{t-z}} x^{-1/2-H} (x+1)^{-1/2-H}
		\left(\left(1-\frac zt\right)x+1\right)^{2H-1}~dx \nonumber\\
		&\leq t^{H-1/2}z^{1/2-H} \left(t-z\right)^{-2H} \int_0^{\infty} x^{-1/2-H} (x+1)^{-1/2-H}~dx\nonumber\\
		&\leq C t^{H-1/2}z^{1/2-H} \left(t-z\right)^{-2H}.\label{eq:Kzt1}
		\end{align}
		Similarly, for $t < z \leq T$, using ${u}-z = \left( z -t\right)x$ yields
		\begin{align}
		\kappa(z,t)   &= (tz)^{1/2-H}\int_z^T (u-t)^{-1/2-H} u^{2H-1} (u-z)^{-1/2-H}~du\nonumber\\
		&= z^{2H-1}(tz)^{1/2-H}\left(  z  -t\right)^{-2H}\times\nonumber\\ &\qquad\int_0^{\frac{T-z}{z-t}} x^{-1/2-H} (x-1)^{-1/2-H}
		\left(\left(1-\frac tz  \right)x+1\right)^{2H-1}~dx\nonumber\\
		&\leq C z^{H-1/2}t^{1/2-H} \left(z-t\right)^{-2H}.\label{eq:Kzt2} 	
		\end{align}
		\item[$(iii,a)$] It follows immediately from \eqref{eq:Kzt1} and \eqref{eq:Kzt2} that for any $ r<\frac1{2H} $,
		\begin{align*} &\int_0^T \kappa(z,t) ^rdz   \\
		&\leq ~C\bigg( \int_0^t \left(t^{H- 1/2}    z^{1/2-H}  (t-z)^{-2H}\right)^{r} dz+\int_t^T  \left(t^{1/2- H}    z^{H-1/2} (z-t)^{-2H}\right)^r dz \bigg)\nonumber\\
		&=~C\left(  t^{rH-r/2}\int_0^t z^{r/2- rH}     (t-z)^{-2Hr} dz + t^{r/2- rH}\int_t^T  z^{rH-r/2}    (z-t)^{-2Hr} dz \right)\nonumber\\
		&\leq C\left(\mathcal{B}(r/2- rH+1, 1-2Hr)  t^{1-2rH} +  \int_t^{T }   (z-t)^{-2Hr} dz\right)
		\leq C.
		 \end{align*}
		Here $\mathcal{B}(\cdot,\cdot)$ is the beta function, and also we used the fact that for $H<1/2$   and for $z\geq t$   we have that $z^{rH-r/2}\leq t^{rH-r/2}$
		\item[$(iii,b)$] It is an immediate consequence of $(iii,a)$.
		\item[$(iv)$]   Now we prove that the kernel $ {\kappa}(z,t)$ is non-negative  definite on $ L_2([0,T])$. First, let us see that for  $ x\in L_2([0,T])$ the value $\int_0^T\int_0^T {\kappa}(z,t)x(z)x(t)dzdt$ is correctly defined. Indeed,  let $ x\in L_2([0,T])$, $1<r<\frac1{2H}$ and denote $y(z)=\int_0^T\kappa(z,t)|x(t)|dt.$ Then, taking into account that $2-r<r$, we get
		\begin{equation*}\begin{gathered}y(z) \leq \int_0^T \left(\kappa(z,t)^r|x(t)|^2\right)^{1/2}\kappa(z,t)^{1-r/2}dt
		 \leq \left(\int_0^T  \kappa(z,t)^r|x(t)|^2 dt\right)^{1/2}\\ \times \left(\int_0^T  \kappa(z,t)^{2-r}dt\right)^{1/2}
	 \leq C\left(\int_0^T  \kappa(z,t)^r|x(t)|^2 dt\right)^{1/2}.
		\end{gathered}\end{equation*}
		Further,
		\begin{equation*}\begin{gathered}\int_0^T\int_0^T {\kappa}(z,t)|x(z)||x(t)|dzdt=\int_0^T|x(z)|y(z)dz\\
		 \leq C\int_0^T\left(|x(z)|\left(\int_0^T  \kappa(z,t)^r|x(t)|^2 dt\right)^{1/2}\right)dz\\
	 \leq C\left(\int_0^T |x(z)|^2dz\right)^{1/2}   \left(\int_0^T \int_0^T  \kappa(z,t)^r|x(t)|^2 dtdz\right)^{1/2}
		 \leq C\|x\|^2_{L_2([0,T])}.
		\end{gathered}\end{equation*}
		Consequently,  $\int_0^T\int_0^T {\kappa}(z,t)x(z)x(t)dzdt$ is well-defined. Then   for any $x\in L_2([0,T])$ it follows from Fubini theorem  that
		\begin{align*}&\int_0^T\int_0^T {\kappa}(z,t)x(z)x(t)dzdt \\
		& = \int_0^T\int_0^t {\kappa}(z,t)x(z)x(t)dtdz+\int_0^T\int_t^T {\kappa}(z,t)x(z)x(t)dtdz\nonumber\\
		 &~=\int_0^Tu^{2H-1}du\int_0^udt\left((u-t)^{-1/2-H}x(t)t^{1/2-H}\right)\int_0^tdz\left((u-z)^{-1/2-H}x(z)z^{1/2-H}\right) \nonumber\\
		 &~~+\int_0^Tu^{2H-1}du\int_0^udz\left((u-z)^{-1/2-H}x(z)z^{1/2-H}\right)\int_0^zdt\left((u-t)^{-1/2-H}x(t)t^{1/2-H}\right) \nonumber\\
		&=\int_0^Tu^{2H-1}		\left(\int_0^u(u-z)^{-1/2-H}x(z)z^{1/2-H}dz\right)^2du\geq 0.
		 	\end{align*}
	\end{itemize}
\hfill$\Box$

\textbf{Proof of Theorem \ref{prop:solution3.5}.}
  In terms of operator $A$, the integral equation \eqref {Fredholm} can be written as $(A+\mathbf{I})x=g^\prime$, where $\mathbf{I}$ is the identical operator. Since operator $A$ is self-adjoint and non-negative definite on $L_2([0,T])$ and consequently has only nonnegative eigenvalues, $Ker(A+\mathbf{I})=Ker(A^*+\mathbf{I})=\{0\}$, then Fredholm alternative states that equation $(A+\mathbf{I})x=g^\prime$ has a unique solution in $L_2([0,T])$.
\hfill$\Box$

\textbf{Proof of Theorem \ref{prop:sufficient}.}
 From Theorem \ref{thm:thm2almeida} (see Appendix \ref{app:A2}), the candidate minimizer obtained in Theorem  \ref{prop:solution3.5} minimizes \eqref{eq:minimization} if $L(t,x,y)$, given by \eqref{eq:Lfct} satisfies
	\begin{align*}
	L(t,x+x_1,y+y_1) - L(t,x,y) \geq \partial_2 L(t,x,y)x_1 + \partial_3 L(t,x,y)y_1
	\end{align*}
	for all $(t,x,y)$, $(t,x+x_1,y+y_1)$ in $[0,T]\times \R^2$, where $\partial_2$ ($\partial_3$) stand, respectively, for the differentiation in $x$ (in $y$).
	Using \eqref{eq:Lfct}, the condition is equivalent to
	\begin{equation*}
	x_1^2 + 2C^{-2}_1 t^{2H-1} y_1^2 \geq 0,
	\end{equation*}
	which is satisfied for any $(x_1,y_1) \in \R^2$.
\hfill$\Box$

\bibliographystyle{econometrica}
\bibliography{Ref_fbm_2}

\end{document}